\documentclass[10.9pt,a4paper]{amsart}
\usepackage{a4wide}

\usepackage[utf8]{inputenc}
\usepackage[english]{babel}

\usepackage{amsfonts,amssymb,amsmath,pinlabel,array,hhline}
\usepackage[nobysame,alphabetic, initials]{amsrefs}

\usepackage{url}
\usepackage{breakurl}
\usepackage[dvipsnames]{xcolor}
\usepackage{xcolor}
\usepackage[colorlinks,
    linkcolor={red!50!black},
    citecolor={blue!50!black},
    urlcolor={blue!80!black},breaklinks={true}]{hyperref}

\usepackage[position=b]{subcaption}

\usepackage{tikz}
\usepackage[all]{xy}
\usepackage{enumitem}
\makeatletter
\newcommand{\mylabel}[2]{#2\def\@currentlabel{#2}\label{#1}}
\usetikzlibrary{calc, matrix, arrows, cd}

\newcommand{\tmfrac}[2]{\mbox{\large$\frac{#1}{#2}$}} 
\newcommand{\unaryminus}{\scalebox{0.75}[1.0]{\( - \)}}

\newcommand{\bsm}{\left(\begin{smallmatrix}}
\newcommand{\esm}{\end{smallmatrix}\right)}

\newtheorem{theorem}{Theorem}[section]

\newtheorem{lemma}[theorem]{Lemma}
\newtheorem{proposition}[theorem]{Proposition}

\newtheorem{question}[theorem]{Question}

\theoremstyle{definition}
\newtheorem{definition}[theorem]{Definition}

\newtheorem{remark}[theorem]{Remark}
\newtheorem{construction}[theorem]{Construction}

\newtheorem*{claim*}{Claim}

\newcommand{\Z}{\mathbb{Z}}
\newcommand{\Q}{\mathbb{Q}}

\newcommand{\proj}{\operatorname{proj}}
\newcommand{\id}{\operatorname{id}}

\begin{document}
\title{A criterion for double sliceness.}
\author{Anthony Conway}
\address{The University of Texas at Austin, Austin TX 78712}
\email{anthony.conway@austin.utexas.edu}
\begin{abstract}
We describe a condition involving noncommutative Alexander modules which ensures that a knot with Alexander module $\Z[t^{\pm 1}]/(t\unaryminus 2) \oplus \Z[t^{\pm 1}]/(t^{-1}\unaryminus 2)$ is topologically doubly slice. 
As an application,  we show that a satellite knot $R_\eta(K)$ is doubly slice if the pattern~$R$ has Alexander module $\Z[t^{\pm 1}]/(t\unaryminus 2) \oplus \Z[t^{\pm 1}]/(t^{-1}\unaryminus 2)$ and satisfies this condition, and if the infection curve $\eta \subset S^3 \setminus R$ lies in the second derived subgroup~$\pi_1(S^3 \setminus R)^{(2)}.$
\end{abstract}

\def\subjclassname{\textup{2020} Mathematics Subject Classification}
\expandafter\let\csname subjclassname@1991\endcsname=\subjclassname
\maketitle

\section{Introduction}

A knot is \emph{slice} if it bounds a locally flatly embedded disc in the~$4$-ball.
The construction of slice discs
typically involves performing band moves on a knot, resulting in a ribbon disc.
In the topological category,  applications of Freedman's disc embedding theorem~\cite{Freedman,FreedmanQuinn,DETBook} lead to other less constructive results: for example, Freedman proved that Alexander polynomial one knots are slice~\cite{Freedman}, and Friedl-Teichner found a condition guaranteeing the sliceness of certain knots with Alexander polynomial~$(t \unaryminus 2)(t^{-1}\unaryminus 2)$~\cite{FriedlTeichner}.
For constructions of slice links, we refer to~\cite{FreedmanNewTechnique,
FreedmanWhitehead,FreedmanTeichner,CochranFriedlTeichner,
ChaPowell,
Manchester}.

Alexander polynomial one knots are in fact \emph{doubly slice}, meaning that they arise as the equatorial cross-section of an unknotted locally flat~$2$-sphere in $S^4$: this can be seen by doubling Freedman's disc and using that a $2$-knot $S \subset S^4$ with knot group $\Z$ is topologically unknotted~\cite{Freedman}.
The aim of this article is to describe a criterion for a knot~$K$ with Alexander module~$\Z[t^{\pm 1}]/(t\unaryminus 2) \oplus \Z[t^{\pm 1}]/(t^{-1}\unaryminus 2)$ to be doubly slice (Theorem~\ref{thm:MainIntro}).
We then investigate to what extent this condition is necessary (Theorem~\ref{thm:Converse}) and apply it to
satellite knots 
(Theorem~\ref{thm:NewDoublySliceIntro}).
Our results rely heavily on~\cite{FriedlTeichner}.
We work in the topological category.

\subsection{Necessary and sufficient conditions for double sliceness}

Given a knot~$K \subset S^3$, write~$M_K$ for the~$3$-manifold obtained by~$0$-surgery on~$K$,~$\pi:=\pi_1(M_K)$ for the fundamental group of~$M_K$, and~$\mathcal{A}_K$ for the Alexander module of~$K$, i.e. the first homology group of the infinite cyclic cover of~$M_K$.
We also use $\pi^{(0)}:=\pi$ and $\pi^{(i+1)}:=[\pi^{(i)},\pi^{(i)}]$ to denote the derived series of~$\pi$.
The groups~$\mathcal{A}_K \cong \pi^{(1)}/\pi^{(2)}$ and~$H_1(M_K)=\pi/\pi^{(1)}$ fit into the short exact sequence
\begin{equation}
\label{eq:SESIntro}
1 \to  \mathcal{A}_K  \to \pi/\pi^{(2)} \to H_1(M_K) \to 1.
\end{equation}
Since~$H_1(M_K) \cong \Z$ is free, this sequence splits,  leading to a group isomorphism~$\pi/\pi^{(2)} \cong \Z \ltimes \mathcal{A}_K$.

Set~$\Lambda:=\Z[t^{\pm 1}]$ and assume that $\mathcal{A}_K \cong \Lambda/(t\unaryminus 2) \oplus \Lambda/(t^{-1}\unaryminus 2)$.
If $P \subset\mathcal{A}_K$ is one of the two summands, then~$\Z \ltimes \mathcal{A}_K/P  \cong \Z \ltimes \Z[\tmfrac{1}{2}]$ is isomorphic, as a group,  to the Baumslag-Solitar group
$$G:=\langle a,c \mid aca^{-1}=c^2 \rangle=BS(1,2).$$
Consequently one can associate to the summand~$P \subset\mathcal{A}_K$,  the epimorphism
$$ \phi_P \colon \pi \twoheadrightarrow \pi/\pi^{(2)} \cong \Z \ltimes \mathcal{A}_K \twoheadrightarrow\Z \ltimes \mathcal{A}_K/P \cong G.$$
Write~$H_1(M_K;\Z[G]_{\phi_P})$ for the noncommutative Alexander module associated to~$\phi_P$, i.e.  the first homology of the~$G$-cover associated to~$\phi_P$.
As we note in Lemma~\ref{lem:ExtWellDef},  the isomorphism type of this $\Z[G]$-module is independent of the splitting of~\eqref{eq:SESIntro} and of the identification of~$\Z \ltimes \mathcal{A}_K$ with~$G=BS(1,2)$.
It therefore makes sense to say that~$P$ \emph{satisfies the Ext condition} if
$$\operatorname{Ext}^1_{\Z[G]}(H_1(M_K;\Z[G]_{\phi_P}),\Z[G])=0.$$

Our main result gives an algebraic criterion for a knot~$K$ to be doubly slice.

\begin{theorem}
\label{thm:MainIntro}
Let~$K$ be a knot with Alexander module~$\mathcal{A}_K \cong \Lambda/(t\unaryminus 2) \oplus \Lambda/(t^{-1}\unaryminus 2)$.
If both summands of~$\mathcal{A}_K$ satisfy the Ext condition,  then~$K$ is doubly slice.
\end{theorem}

Knots with trivial Alexander module (i.e.  Alexander polynomial one knots) are doubly slice, so Theorem~\ref{thm:MainIntro} can be thought of as an analogue of this result when the Alexander module is~$\Lambda/(t\unaryminus 2) \oplus \Lambda/(t^{-1}\unaryminus 2)$.
Another way of thinking of Theorem~\ref{thm:MainIntro} is as doubly-slice refinement of Friedl and Teichner's theorem~\cite{FriedlTeichner} that, given a knot $K$, if there exists an epimorphism~$\psi \colon \pi_1(M_K) \twoheadrightarrow~G$ with $\operatorname{Ext}^1_{\Z[G]}(H_1(M_K;\Z[G]_\psi),\Z[G])=0$, then~$K$ is~$G$-homotopy ribbon.
Here a disc $D$ is called \emph{homotopy ribbon} if the inclusion induced map $\pi_1(S^3 \setminus K) \to \pi_1(D^4 \setminus D)$ is surjective and \emph{$G$-homotopy ribbon} if, additionally,  $\pi_1(D^4 \setminus D) \cong G$.
We will refer to $\pi_1(D^4 \setminus D)$ as the (disc) \emph{group of $D$} by analogy with classical knot theory where $\pi_1(S^3 \setminus K)$ is called the (knot) group of $K$.
We emphasise that the proof of Theorem~\ref{thm:MainIntro} relies heavily on~\cite{FriedlTeichner}.

Necessary and sufficient conditions are known for a knot~$K$ to be~$G$-homotopy ribbon~\cite[Corollary 1.7]{ConwayPowellDiscs} and, similarly,  it is natural to wonder whether the \emph{double Ext condition} of Theorem~\ref{thm:MainIntro} is necessary for double sliceness.
Unfortunately this is not the case: Proposition~\ref{prop:NotNecessary} describes a doubly slice knot~$K$ that has Alexander module $\mathcal{A}_K \cong \Lambda/(t\unaryminus 2)  \oplus \Lambda/(t^{-1}\unaryminus 2)$ but that does not satisfy the double Ext condition.
We nevertheless obtain the following converse.

\begin{theorem}
\label{thm:Converse}
Let~$K$ be a knot with Alexander module~$\mathcal{A}_K \cong \Lambda/(t\unaryminus 2)  \oplus \Lambda/(t^{-1}\unaryminus 2)$.
The following statements are equivalent:
\begin{enumerate}
\item The knot~$K \subset S^3$ arises as the equatorial cross-section of a~$2$-sphere~$S \subset S^4=D^4_1 \cup_{S^3} D^4_2$ with $\pi_1(S^4 \setminus S) \cong \Z$,  and the group of the disc~$D_i=D^4_i \cap S$ is metabelian for $i=1,2.$
\item The summands of~$\mathcal{A}_K$ both satisfy the Ext condition.
\end{enumerate}
\end{theorem}

Here recall that a group $\Gamma$ is \emph{metabelian} if its second derived subgroup is trivial, i.e. if~$\Gamma^{(2)}=1$.

\begin{proof}[Proof of Theorem~\ref{thm:MainIntro} assuming Theorem~\ref{thm:Converse}]
Since we assumed that both summands of $\mathcal{A}_K$ satisfy the Ext condition,  Theorem~\ref{thm:Converse} implies that $K$ arises as the equatorial cross-section of a $2$-sphere~$S=D_1 \cup_K D_2 \subset S^4$ that satisfies~$\pi_1(S^4 \setminus S )\cong \Z$.
A result of Freedman now implies that~$S$ is topologically unknotted~\cite{Freedman} thus showing that~$K$ is doubly slice.
\end{proof}

\begin{remark}
\label{rem:Outline}
We outline the proof of Theorem~\ref{thm:Converse}.
\begin{itemize}
\item We start with the~$(2) \Rightarrow (1)$ direction.
Write $\mathcal{A}_{D_i}=H_1(D^4 \setminus D_i;\Lambda)$ for the Alexander module of the~$D_i$ and use $P_1,P_2$ to denote the summands of $\mathcal{A}_K$.
The first part of the proof consists of showing that~$\lbrace P_1,P_2 \rbrace=\lbrace P_{D_1},P_{D_2}\rbrace,$ where we write~$P_{D_i}:=\ker(\mathcal{A}_K \to \mathcal{A}_{D_i}).$
This step does not use the fact that the disc groups are metabelian; see Proposition~\ref{prop:GoalTechnical}.
The metabelian condition is used in the second step to show that the $D_i$ are~$G$-homotopy ribbon; see Proposition~\ref{prop:BS12}.
Once we know that the discs are $G$-homotopy ribbon, it follows from~\cite{FriedlTeichner,ConwayPowellDiscs} that the~$P_{D_i}$ (and thus the $P_i$) satisfy the Ext condition for~$i=1,2$.
\item 
Next we outline the proof of the~$(1) \Rightarrow (2)$ direction.
Since~$P_1$ and~$P_2$ both satisfy the Ext condition,  we apply~\cite{FriedlTeichner,ConwayPowellDiscs} to deduce that the knot~$K$ bounds~$G$-homotopy ribbon discs with $P_{D_i}=P_i$. 
Most of the proof is then devoted to showing that the sphere~$S=D_1 \cup_K D_2$ has knot group~$\Z$.
Note that $S$ is obtained as the union of two distinct discs with boundary $K$ whereas, for Alexander polynomial one knots, the $2$-sphere is obtained as a union of two isotopic discs.
\end{itemize}
\end{remark}

In particular, the proof of Theorem~\ref{thm:Converse} shows that if $D_1,D_2 \subset D^4$ are two $G$-homotopy ribbon discs with boundary $K$ and $P_{D_1} \neq P_{D_2}$,  then the $2$-sphere $D_1 \cup_K D_2 \subset S^4$ is topologically unknotted.
This leads to the following question in the smooth category.
\begin{question}
Let~$K$ be a knot with Alexander module~$\mathcal{A}_K \cong \Lambda/(t\unaryminus 2)  \oplus \Lambda/(t^{-1}\unaryminus 2)$.
If~$D_1,D_2 \subset~D^4$ are smoothly embedded $G$-homotopy ribbon discs with boundary~$K$ and~$P_{D_1} \neq P_{D_2}$, is the topologically unknotted sphere~$S=D_1 \cup_K D_2$ necessarily smoothly unknotted?
\end{question}


\subsection{An application to doubly slicing  satellite knots }
\label{sub:ApplicationIntro}

Let~$R \sqcup \eta \subset S^3$ be a~$2$-component link with~$\eta$ an unknot and~$R$ contained in the interior of the solid torus~$V=S^3 \setminus~\nu(\eta)$.
Let~$K \subset S^3$ be another knot and let~$h \colon V \to \overline{\nu}(K) \subset S^3$ be an orientation preserving homeomorphism taking the core~$c \subset V$ to~$K \subset \overline{\nu}(K)$ and a~$0$-framed longitude of~$c$ to a~$0$-framed longitude of~$K$.
The image of~$R$ under~$h$, denoted~$R_\eta(K)$, is the \emph{satellite knot} with \emph{pattern}~$R$,  \emph{companion}~$K$,  \emph{infection curve}~$\eta$ and \emph{winding number}~$n=\ell k(R,\eta)$.
When we write ``let~$R$ be a pattern and~$\eta \subset S^3 \setminus R$ be an infection curve",  it is understood that $\eta$ is unknotted in $S^3.$
We will sometimes also say that~$R_\eta(K)$ is obtained by \emph{infecting} $R$ along $\eta$ as e.g. in~\cite{CochranOrrTeichnerStructure}.

If~$R$ and~$K$ are both doubly slice, then~$R_\eta(K)$ is known to be doubly slice, see e.g.~\cite[Proposition~3.4]{Meier}.
We apply Theorem~\ref{thm:MainIntro} to show that when $\mathcal{A}_R \cong \Lambda/(t\unaryminus 2) \oplus \Lambda/(t^{-1}\unaryminus 2)$,  there is a criterion on $R$ which ensures that~$R_\eta(K)$ is doubly slice for \emph{any} knot~$K$.

\begin{theorem}
\label{thm:NewDoublySliceIntro}
Let~$R$ be a pattern and let~$\eta \subset S^3 \setminus R$ be an infection curve that lies in~$\pi_1(S^3 \setminus R)^{(2)}$.
If both summands of~$\mathcal{A}_{R}=\Lambda/(t\unaryminus 2) \oplus \Lambda/(t^{-1}\unaryminus 2)$ satisfy the Ext condition,  then~$R_\eta(K)$ is doubly slice for any knot~$K$.
\end{theorem}

We describe an application of this result.
Consider the knot $\mathcal{R}^{T_1,T_2}$ obtained from $\mathcal{R}=9_{46}$ by satellite operations along the infections curves $\gamma_1,\gamma_2 \subset S^3  \setminus  \mathcal{R}$ illustrated in Figure~\ref{fig:946Satellite} and companions untwisted Whitehead doubles~$\operatorname{Wh}(T_1)$ and~$\operatorname{Wh}(T_2)$; either choice of clasp is acceptable.
We denote by $\mathcal{R}^{T_1,T_2}(K_1,K_2)$ the outcome of performing two additional satellite operations along the curves $\eta_1,\eta_2$ illustrated in Figure~\ref{fig:946Satellite} with companions~$K_1$ and~$K_2$.
Proposition~\ref{prop:Application} applies Theorem~\ref{thm:NewDoublySliceIntro} to show that $\mathcal{R}^{T_1,T_2}(K_1,K_2)$ is doubly slice for any $T_1,T_2,K_1,K_2$.
It seems likely that there are choices of $T_1,T_2,K_1,K_2$ that ensure that~$\mathcal{R}^{T_1,T_2}(K_1,K_2)$ is not smoothly slice (and thus not smoothly doubly slice), but we will not pursue this question further.

\begin{figure}[!htbp]
\centering
\captionsetup[subfigure]{}
\begin{subfigure}{0.45\textwidth}
\includegraphics[scale=0.3]{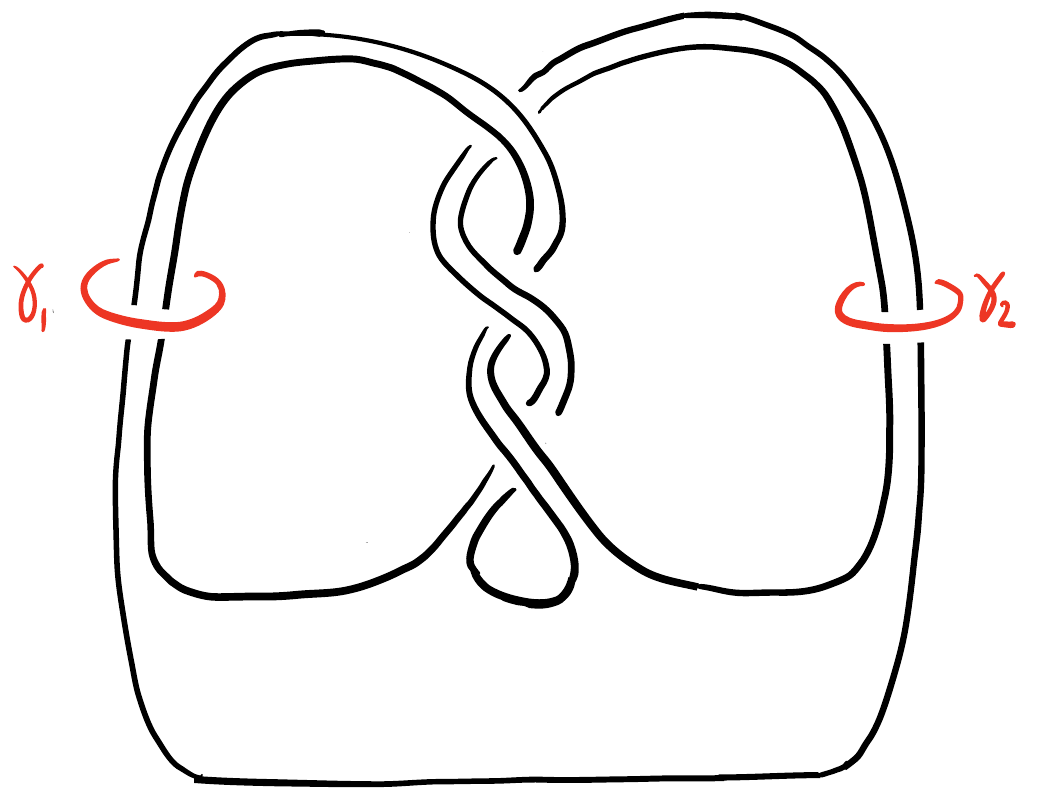}
\end{subfigure}
 \hspace{1cm}
\begin{subfigure}{.45\textwidth}
\includegraphics[scale=0.3]{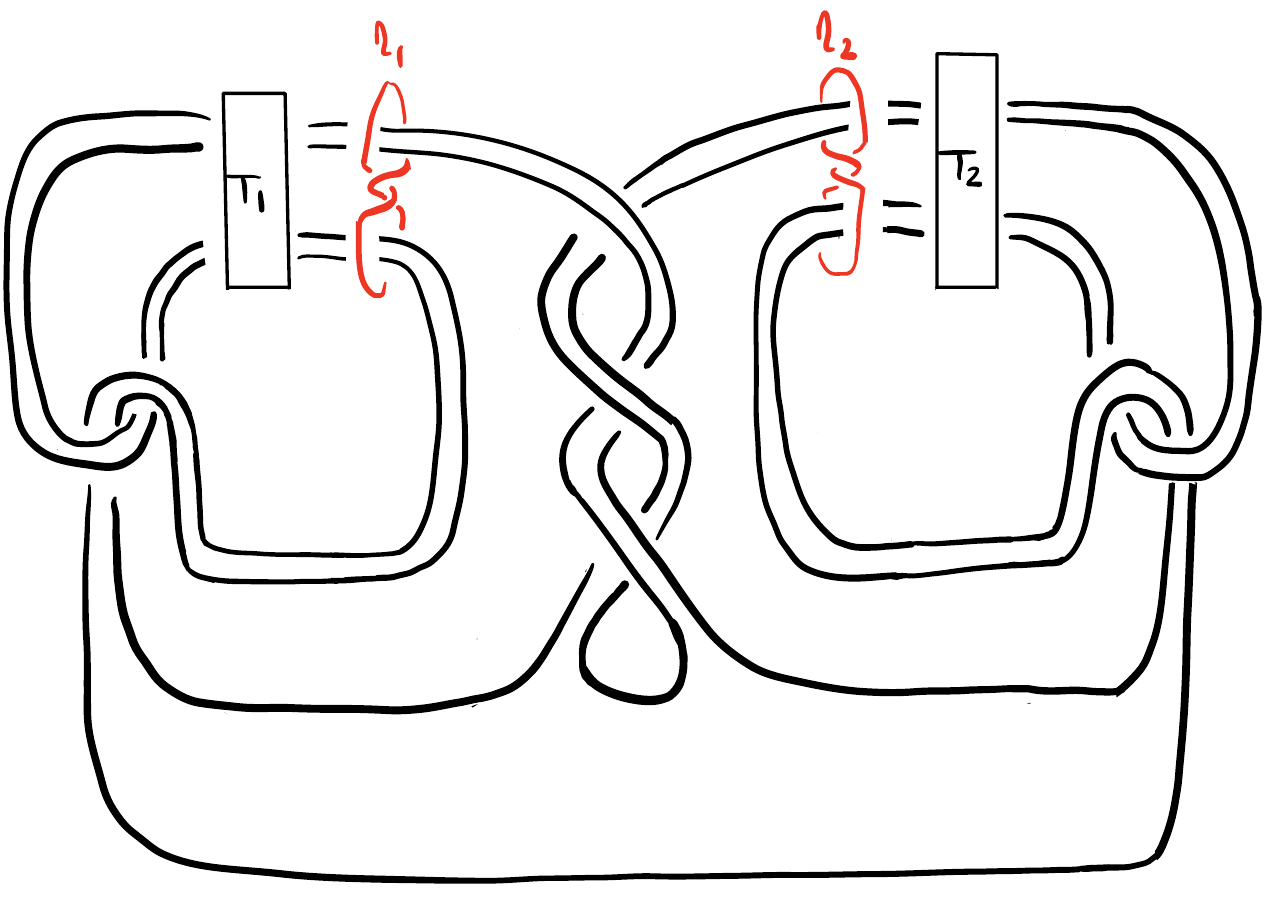}
\end{subfigure}
\caption{
On the left: the knot $\mathcal{R}=9_{46}$ together with two infections curves~$\gamma_1,\gamma_2 \subset S^3 \setminus \mathcal{R}$.
On the right: the knot $\mathcal{R}^{T_1,T_2}$ with the infection curves~$\eta_1,\eta_2$ that we use to produce $\mathcal{R}^{T_1,T_2}(K_1,K_2)$.}
\label{fig:ExampleDoublySlice}
\end{figure}

\begin{remark}
\label{rem:CFT}
It is likely that Theorem~\ref{thm:NewDoublySliceIntro} can be also be proved by applying~\cite{CochranFriedlTeichner} as follows.
The conditions on $R$ ensure that it bounds $G$-homotopy ribbon discs $D_1$  and~$D_2$~\cite{FriedlTeichner}.
Since~$\eta$ lies in~$ \pi_1(S^3 \setminus R)^{(2)}$ and $G$ is metabelian, $\eta$ is trivial in~$\pi_1(D^4 \setminus D_i) \cong G$.
Applying~\cite[Corollary~1.6]{CochranFriedlTeichner} to $D_1$ and $D_2$ leads to slice discs~$D_1'$ and $D_2'$ for~$R_\eta(K)$.
It seems likely that a close inspection of~\cite[proof of Theorem~1.5]{CochranFriedlTeichner} would show that the $D_i'$ are still~$G$-homotopy ribbon and that~$D_1' \cup_{R_\eta(K)} D_2'$ has knot group~$\Z$ and therefore (by Freedman~\cite{Freedman}) doubly slices $R_\eta(K)$.
I am grateful to Mark Powell for suggesting this.
It would be interesting to explore whether other applications of~\cite{CochranFriedlTeichner} lead to further constructions of doubly slice satellite links.
\end{remark}

\subsection*{Organisation}
Section~\ref{sec:Background} contains some background results on twisted homology,  the Ext condition and $G$-homotopy ribbon discs.
Section~\ref{sec:MainProof} is concerned with the proof of Theorem~\ref{thm:Converse} and Section~\ref{sec:Application} proves Theorem~\ref{thm:NewDoublySliceIntro}.

\subsection*{Conventions}

Throughout this article, we work in the topological category. 
Manifolds are assumed to be compact, connected, based and oriented; if a manifold has a nonempty boundary, then the basepoint is assumed to be in the boundary.
We denote the exterior of a knot $K \subset S^3$ by~$E_K:=S^3 \setminus \nu(K).$
We also set $\Lambda:=\Z[t^{\pm 1}]$ and $G:=BS(1,2)$.
Given Laurent polynomials~$p,q \in \Lambda$ we write $p \doteq q$ if $p=\pm t^k q$ for some $k \in \Z$.

\subsection*{Acknowledgments}

This research was partially supported by the NSF grant DMS\unaryminus 2303674.
Thanks go to Mark Powell for comments on a draft of this paper and to Lisa Piccirillo for heplful discussions. 
Finally, I am grateful to an anonymous referee for helpful comments and in particular for a suggestion that streamlined the proof of Theorem~\ref{thm:VanKampen}.

\section{Background}
\label{sec:Background}

This short section collects some background results on twisted homology, the Ext condition and the classification of $G$-homotopy ribbon discs.
Further references on twisted homology include~\cite{FriedlLectureNotes,
FriedlNagelOrsonPowell,KirkLivingston}, whereas we refer to~\cite{FriedlTeichner,ConwayPowellDiscs,ConwayDiscs} for more context surrounding the Ext condition.

\subsection{Twisted homology}
\label{sub:TwistedHomology}

In what follows,  spaces are assumed to have the homotopy type of a finite CW complex.
Given a space~$X$ and an epimorphism~$\psi \colon \pi_1(X) \twoheadrightarrow \Gamma$,  we write~$X_\psi$ for the~$\Gamma$-cover associated to~$\ker(\psi)$ and~$H_*(X;\Z[\Gamma]_\psi):=H_*(X_\psi)$.
The left action of~$\Gamma$ on~$X_\psi$ by deck transformations endows~$H_*(X;\Z[\Gamma]_\psi)$ with the structure of a left~$\Z[\Gamma]$-module.
Alternatively~$H_*(X;\Z[\Gamma]_\psi)$ can be defined as the homology of~$\Z[\Gamma]_\psi \otimes_{\Z[\pi_1(X)]} C_*(\widetilde{X})$, a chain complex that is chain isomorphic to~$C_*(X_\psi)$.
Here we used $\widetilde{X} \to X$ to denote the universal cover of $X$ and the right~$\Z[\pi_1(X)]$-module structure on~$\Z[\Gamma]_\psi$ is given by~$\gamma \cdot g=\gamma \psi(g)$ for $\gamma \in \Gamma$ and $g \in \pi_1(X)$.
It is this latter definition of~$H_*(X;\Z[\Gamma]_\psi)$ that we will be using most frequently.

\begin{lemma}
\label{lem:TwistedHomology}
Given a space $X$  and an epimorphism $\psi \colon \pi_1(X) \twoheadrightarrow \Gamma$, a group automorphism~$\varphi \colon \Gamma \to~\Gamma$ induces a~$\Z[\Gamma]$-isomorphism
$$H_*(X;\Z[\Gamma]_\psi ) \xrightarrow{\cong}  H_*(X;\Z[\Gamma]_{\varphi \circ \psi}).$$
\end{lemma}
\begin{proof}
The lemma follows once one verifies that~$\varphi$ induces a well defined chain isomorphism
$$\Z[\Gamma]_\psi \otimes_{\Z[\pi_1(X)]} C_*(\widetilde{X})  \to \Z[\Gamma]_{\varphi \circ \psi} \otimes_{\Z[\pi_1(X)]} C_*(\widetilde{X}).$$
This chain map is well defined because~$\gamma  \otimes g \sigma$  and~$\gamma \psi(g) \otimes \sigma$ are mapped to the same element, namely~$\varphi(\gamma) \varphi(\psi(g)) \otimes \sigma$.
\end{proof}

\subsection{The Ext condition}
\label{sub:Ext}

Given a knot~$K$, write~$M_K$ for the closed~$3$-manifold obtained by~$0$-framed surgery on~$K$,~$\pi:=\pi_1(M_K)$ for the fundamental group of~$M_K$,  $\operatorname{ab}\colon \pi_K \twoheadrightarrow H_1(M_K) \cong \Z$ for abelianisation, and~$\mathcal{A}_K=H_1(M_K;\Lambda)$ for the Alexander module of~$K$.
Writing $\pi^{(0)}:=\pi$ and~$\pi^{(k)}:=[\pi^{(k-1)},\pi^{(k-1)}]$ when $k \geq 1$ for the derived series of~$\pi$,  recall from the introduction that the groups~$\mathcal{A}_K \cong \pi^{(1)}/\pi^{(2)}$ and~$H_1(M_K) \cong \pi/\pi^{(1)}$ fit into the short exact sequence
\begin{equation}
\label{eq:ExactSequencepi2}
1 \to  \mathcal{A}_K  \to \pi/\pi^{(2)} \xrightarrow{\operatorname{ab}} H_1(M_K) \to 1.
\end{equation}
Since~$H_1(M_K) \cong \Z$ is free, this short exact sequence splits.
The choice of a splitting~$\theta \colon \Z \to~\pi/\pi^{(2)}$ amounts to the choice of (the homotopy class of) a meridian $\mu_K \in \pi_1(M_K)$ as one can then set~$\theta(n):=[\mu_K^n]$.
One then verifies that~$\vartheta \colon \pi/\pi^{(2)} \xrightarrow{\cong} \Z \ltimes \mathcal{A}_K,  g \mapsto (\operatorname{ab}(g),[g\mu_K^{-\operatorname{ab}(g)}])$
is an isomorphism.
Here, the action of $\Z$ on~$\mathcal{A}_K \cong \pi^{(1)}/\pi^{(2)}$ is by~$n \cdot x=\mu_K^n x \mu_K^{-n}$ where~$n \in \Z$ and~$x \in \mathcal{A}_K$.
When we consider the semidirect product $\Z \ltimes \mathcal{A}_K$, it is with respect to this action.

Given a submodule~$P \leq \mathcal{A}_K$,  write~$\Gamma$ for the group underlying~$\Z \ltimes \mathcal{A}_K/P$.
The choice of a splitting of~\eqref{eq:ExactSequencepi2} and of an identification of~$\Z \ltimes \mathcal{A}_K/P$ with~$\Gamma$ gives rise to an epimorphism
$$ \phi_P \colon \pi \twoheadrightarrow \pi/\pi^{(2)} \xrightarrow{\cong,\vartheta} \Z \ltimes \mathcal{A}_K \twoheadrightarrow \Z \ltimes \mathcal{A}_K/P \cong \Gamma.$$
The next lemma will allow us to define the Ext condition mentioned in the introduction.
\begin{lemma}
\label{lem:ExtWellDef}
Given a knot $K$ and a submodule $P \leq \mathcal{A}_K$, the isomorphism type of the $\Z[\Gamma]$-module~$H_1(M_K;\Z[\Gamma]_{\phi_P})$ depends  neither on the splitting of~\eqref{eq:ExactSequencepi2} used to identify~$\pi/\pi^{(2)}$ with ~$\Z \ltimes \mathcal{A}_K$ nor on the automorphism of~$\Gamma$ used to identify~$\Z \ltimes \mathcal{A}_K/P$ with~$\Gamma$.

In particular the condition
$$ \operatorname{Ext}^1_{\Z[\Gamma]}(H_1(M_K;\Z[\Gamma]_{\phi_P}),\Z[\Gamma])=0$$
only depends on the knot~$K$ and on the submodule~$P$.
\end{lemma}
\begin{proof}
Let~$\theta_1,\theta_2 \colon \Z \to \pi/\pi^{(2)}$ be two splittings of~\eqref{eq:ExactSequencepi2} and write~$\vartheta_1,\vartheta_2 \colon \pi/\pi^{(2)} \to \Z \ltimes \mathcal{A}_K$ for the resulting isomorphisms.
It is proved in~\cite[Lemma 3.5]{ParkPowell} that there exists an automorphism~$\varphi$ of~$\Z \ltimes \mathcal{A}_K$ that restricts to the identity on~$\mathcal{A}_K$ and makes the following square commute:
$$
\xymatrix@R0.5cm@C0.5cm{
\pi/\pi^{(2)} \ar[r]^-{\vartheta_1} \ar[d]^-=&  \Z \ltimes \mathcal{A}_K \ar[d]^{\varphi}  \\
\pi/\pi^{(2)} \ar[r]^-{\vartheta_2}&  \Z \ltimes \mathcal{A}_K.
}
$$
Since~$\varphi$ restricts to the identity on~$\mathcal{A}_K$, it takes~$P$ to itself and therefore descends to an isomorphism~$\Z \ltimes \mathcal{A}_K/P \to  \Z \ltimes \mathcal{A}_K/P$.
A choice of isomorphisms~$\nu_1,\nu_2 \colon  \Z \ltimes \mathcal{A}_K/P  \cong \Gamma$ then leads to an automorphism $\varphi_\Gamma \colon \Gamma \to \Gamma$ that makes the following diagram commute:
$$
\xymatrix@R0.5cm{
\pi  \ar[r]\ar[d]^= \ar@/^2pc/[rrrr]^-{\phi_P}& \pi/\pi^{(2)} \ar[r]_-{\vartheta_1,\cong} \ar[d]^=&  \Z \ltimes \mathcal{A}_K \ar[r]\ar[d]^{\varphi,\cong}& \Z \ltimes \mathcal{A}_K/P \ar[r]_-{\nu_1,\cong}\ar[d]^{\varphi,\cong}& \Gamma \ar[d]^{\varphi_\Gamma,\cong} \\
\pi  \ar[r] \ar@/_2pc/[rrrr]_-{\phi_P'}& \pi/\pi^{(2)} \ar[r]^-{\vartheta_2,\cong} &  \Z \ltimes \mathcal{A}_K \ar[r]& \Z \ltimes \mathcal{A}_K/P \ar[r]^-{\nu_2,\cong}& \Gamma.
}
$$
Since the automorphism~$\varphi_\Gamma \colon \Gamma\to \Gamma$ satisfies~$\varphi_\Gamma \circ \phi_P=\phi_{P}'$, the first assertion now follows from Lemma~\ref{lem:TwistedHomology}: $H_ 1(M_K;\Z[\Gamma]_{\phi_P})  \cong  H_ 1(M_K;\Z[\Gamma]_{\varphi_\Gamma \circ \phi_P}) =H_ 1(M_K;\Z[\Gamma]_{\phi_P'}).$
The second assertion follows from the first because the Ext functor is contravariant in its first variable.
\end{proof}

Lemma~\ref{lem:ExtWellDef} justifies the following definition.

\begin{definition}
\label{def:Ext}
Given a knot~$K$, a submodule~$P \subset \mathcal{A}_K$ \emph{satisfies the Ext condition} if 
$$ \operatorname{Ext}^1_{\Z[\Gamma]}(H_1(M_K;\Z[\Gamma]_{\phi_P}),\Z[\Gamma])=0,$$
where~$\Gamma:=\Z \ltimes \mathcal{A}_K/P$.
\end{definition}

\subsection{The classification of~$G$-homotopy ribbon discs}
\label{sub:Classification}

Set $G:=BS(1,2)$.
The combination of~\cite{FriedlTeichner} and~\cite{ConwayPowellDiscs} gives the classification of~$G$-homotopy ribbon discs for a knot~$K$ up to isotopy rel. boundary.
When~$K$ has Alexander module~$\mathcal{A}_K \cong \Lambda/(t\unaryminus 2) \oplus \Lambda/(t^{-1}\unaryminus 2)$~\cite[Theorem~1.6 and Remark 1.8]{ConwayPowellDiscs} proves that~$G$-homotopy-ribbon discs for~$K$ are in bijective correspondence with the elements of~$\lbrace P_1,P_2\rbrace :=\lbrace \Lambda/(t\unaryminus 2) \oplus 0,0 \oplus \Lambda/(t^{-1}\unaryminus 2) \rbrace$ that satisfy the Ext condition.
This set can have cardinality~$0,1$ or~$2$: all three scenarios arise~\cite[Subsection 1.4]{ConwayPowellDiscs}.

The bijection is explicit: it maps a~$G$-homotopy-ribbon disc~$D \subset D^4$ with boundary~$K$ to the submodule~$P_D=\ker(\mathcal{A}_K \to \mathcal{A}_D)$ where~$\mathcal{A}_D=H_1(D^4 \setminus \nu(D);\Lambda)$ denotes the Alexander module of~$D$.
Setting $N_D:=D^4 \setminus \nu(D)$, an equivalent perspective (taken in~\cite{ConwayDiscs}) is to think of the bijection as mapping $D$ to the inclusion induced map $\pi_1(M_K) \twoheadrightarrow \pi_1(N_D) \cong G$ considered up to post-composition with an automorphisms of~$G$.
In particular, for~$P_i$ as above,  there exists a~$G$-homotopy ribbon disc~$D_i \subset D^4$ with boundary~$K$ such that~$\phi_{P_i}$ and~$\pi_1(M_K) \twoheadrightarrow \pi_1(N_D)$ agree up to post-composition with an automorphism of~$G$.


\section{The necessary and sufficient conditions}
\label{sec:MainProof}

The goal of this section is to prove Theorem~\ref{thm:Converse} from the introduction.
In Subsection~\ref{sub:Lemmas} we collect some technical results that will be used in Subsection~\ref{sub:ProofMain} to prove Theorem~\ref{thm:Converse}.

\subsection{Preliminary lemmas}
\label{sub:Lemmas}

This section records two technical results (Propositions~\ref{prop:GoalTechnical} and~\ref{prop:BS12}) that will be needed to prove the $(2) \Rightarrow (1)$ direction of Theorem~\ref{thm:Converse}.
Throughout this section,  we write $N_D:=D^4 \setminus \nu(D)$ and $X_S:=S^4 \setminus \nu(S)$ for the exterior of discs and spheres respectively.
We also continue writing $\Lambda:=\Z[t^{\pm 1}]$ and $G:=BS(1,2)$.

\medbreak

Following~\cite{ParkPowell}, we say that a disc $D$ is \emph{$\Lambda$-homology ribbon} if the inclusion induced homomorphism~$\mathcal{A}_K \to \mathcal{A}_D$ is surjective.
The following lemma is implicit in~\cite[Proposition 2.10]{Kim} (see also~\cite[Lemma 3.3]{ParkPowell})
 but we provide the argument as it is short.

\begin{lemma}
\label{lem:DoublySlice}
If a knot $K \subset S^3$ is the equatorial cross-section of a sphere $S=D_1 \cup_K D_2 \subset S^4$ with $\pi_1(X_S) \cong \Z$, then 
\begin{enumerate}
\item the discs $D_1$ and $D_2$ are $\Lambda$-homology ribbon;
\item the inclusions $P_{D_1},P_{D_2} \subset \mathcal{A}_K$ induce an isomorphism
$ P_{D_1} \oplus P_{D_2} \xrightarrow{\cong} \mathcal{A}_K.$
 \end{enumerate}
\end{lemma}
\begin{proof}
Since $\pi_1(X_S)\cong \Z$ we deduce that~$H_1(X_S;\Lambda)=0=H_3(X_S;\Lambda)=0$ and that~$H_2(X_S;\Lambda)$ is free~\cite[Lemma 3.2]{ConwayPowell}.
On the other hand, since $\chi(X_S)=0$, it follows that $H_2(X_S;\Lambda)$ is~$\Lambda$-torsion.
We therefore conclude that~$H_2(X_S;\Lambda)=0$.
The Mayer-Vietoris sequence for the decomposition~$X_S=N_{D_1} \cup_{S^3 \setminus \nu(K)}  N_{D_2}$ then shows that the inclusions induce an isomorphism
$$\bsm i_{D_1} \\ i_{D_2} \esm \colon \mathcal{A}_K \xrightarrow{\cong} \mathcal{A}_{D_1} \oplus \mathcal{A}_{D_2}.$$
This implies in particular that $i_{D_1}$ and $i_{D_2}$ are surjective, i.e. that the $D_i$ are $\Lambda$-homology ribbon.
A quick verification then shows that these inclusions induce isomorphisms $i_{D_2} \colon P_{D_1} \xrightarrow{\cong} \mathcal{A}_{D_2}$ and~$i_{D_1} \colon P_{D_2} \xrightarrow{\cong} \mathcal{A}_{D_1}$.
The lemma now follows readily.
\end{proof}

Lemma~\ref{lem:DoublySlice} shows that the properties of $\Lambda$-homology ribbon discs are relevant to the study of doubly slice knots.
We start by describing a condition which ensures that a $\Lambda$-homology ribbon disc is in fact homotopy ribbon.

\begin{proposition}
\label{prop:HomotopyRibbon}
A $\Lambda$-homology ribbon disc $D$ with metabelian disc group is homotopy ribbon.
\end{proposition}
\begin{proof}
Set $\pi_K:=\pi_1(M_K)$ and $\pi_D:=\pi_1(N_D)$,  write $\mathcal{A}_K$ and $\mathcal{A}_D$ for the Alexander modules of $K$ and $D$ respectively,  and consider the following commutative diagram in which the rows are exact:
$$
\xymatrix@R0.5cm@C0.5cm{
1 \ar[r]&\mathcal{A}_K \ar[r]\ar@{->>}[d]& \pi_K/\pi_K^{(2)} \ar[r]^-{\operatorname{ab}}\ar[d]^{\pi_1(\iota)}& \Z\ar[r]\ar[d]_\cong^{H_1(\iota)}&1 \\
1\ar[r]&\mathcal{A}_D \ar[r]& \pi_D/\pi_D^{(2)} \ar[r]^-{\operatorname{ab}}& \Z\ar[r]&1.
}
$$
Both rows split because $\Z$ is free.
A splitting $\theta_K \colon \Z \to \pi_K/\pi_K^{(2)}$ of the top row then defines a splitting~$\theta_D \colon \Z \to \pi_D/\pi_D^{(2)}$ of the bottom row via the formula $\theta_D:=\pi_1(\iota) \circ s_K \circ H_1(\iota)^{-1}$.
More explicitly,  if~$\theta_K(n)=[\mu_K^n]$ for some meridian $\mu_K \in \pi_1(M_K)$, then $\theta_D(n)=[\mu_D^n]$ where~$\mu_D \in \pi_1(N_D)$ is a meridian that satisfies~$\pi_1(\iota)(\mu_K)=\mu_D$.
A rapid verification shows that the resulting isomorphisms~$\vartheta_K \colon \pi_K/\pi_K^{(2)} \to \Z \ltimes \mathcal{A}_K,g \mapsto (\operatorname{ab}(g),[g\mu_K^{-\operatorname{ab}(g)}])$ and~$\pi_D/\pi_D^{(2)} \to \Z \ltimes \mathcal{A}_D,g \mapsto (\operatorname{ab}(g),[g\mu_D^{-\operatorname{ab}(g)}])$ then make the following diagram commute:
$$
\xymatrix@R0.5cm{
\pi_K  \ar@{->>}[r]\ar[d]^{\pi_1(\iota)} & \pi_K/\pi^{(2)}_K \ar[r]^-{\vartheta_K,\cong} \ar[d]^{\pi_1(\iota)}&  \ar@{->>}[d]^{}\Z \ltimes \mathcal{A}_K & \\
\pi_D  \ar[r]^-= & \pi_D/\pi_D^{(2)} \ar[r]^-{\vartheta_D,\cong} &  \Z \ltimes \mathcal{A}_D.&  
}
$$
The bottom left arrow is an equality because we assumed that $\pi_D^{(2)}=1$.
The right arrow is surjective because $D$ is $\Lambda$-homology ribbon.
The commutativity of this diagram implies that the leftmost map labelled $\pi_1(\iota)$ is surjective i.e. that~$D$ is homotopy ribbon.
\end{proof}

The next lemma records some additional facts about $\Lambda$-homology ribbon discs.
\begin{lemma}
\label{lem:HomologyRibbon}
If $D$ is a $\Lambda$-homology ribbon disc, then 
\begin{enumerate}
\item the submodule $P_D \leq \mathcal{A}_K$ fits into the short exact sequence
\begin{equation}
\label{eq:SESPD}
0 \to P_D \to \mathcal{A}_K \to \mathcal{A}_D \to  0;
\end{equation}
\item $P_D$ has projective dimension at most $1$ as a $\Lambda$-module,  and is torsion free as an abelian~group.
\end{enumerate}
\end{lemma}
\begin{proof}
Since $D$ is $\Lambda$-homology ribbon, we deduce that~$H_1(N_D,\partial N_D;\Lambda)=0$.
The first assertion now follows by considering the long exact sequence of the pair $(N_D,\partial N_D)$. 
We prove the second assertion.
Since $\mathcal{A}_K$ is $\Lambda$-torsion and admits a square presentation matrix~\cite{LevineKnotModules}, its projective dimension is at most $1$.
Since $\Lambda$ has global dimension $2$,  we deduce that $\mathcal{A}_D$ has projective dimension at most~$2$.
The exact sequence from~\eqref{eq:SESPD} now implies that $P_D$ has projective dimension at most $1$,  for if $V$ is a $\Lambda$-module, then the exact sequence
$$0\leftarrow  \underbrace{\operatorname{Ext}_\Lambda^3(\mathcal{A}_D,V)}_{=0}  \leftarrow \operatorname{Ext}_\Lambda^2(P_D,V)  \leftarrow \underbrace{\operatorname{Ext}_\Lambda^2(\mathcal{A}_K,V)}_{=0} \leftarrow \operatorname{Ext}_\Lambda^2(\mathcal{A}_D,V) \leftarrow \ldots $$
implies that $\operatorname{Ext}^i_\Lambda(P_D,V)=0$ for $i \geq 2$.
Finally, the fact that $P_D \leq \mathcal{A}_K$ is $\Z$-torsion free follows because~$\mathcal{A}_K$ is $\Z$-torsion free~\cite[Proposition~3.5]{LevineKnotModules}.
\end{proof}

\begin{remark}
It is tempting to conjecture that if a torsion $\Lambda$-module $P$ is $\Z$-torsion free, admits a square presentation matrix and satisfies~$\operatorname{Ord}(P)\doteq  t\unaryminus 2$, then $P \cong \Lambda/(t\unaryminus 2)$.
This would simplify the remainder of this section as well as imply that if a doubly slice knot has Alexander polynomial~$(t\unaryminus2)(t^{-1}\unaryminus2)$, then it necessarily has Alexander module~$\mathcal{A}_K \cong \Lambda/(t\unaryminus 2) \oplus \Lambda/(t^{-1}\unaryminus 2)$.
\end{remark}

\begin{lemma}
\label{lem:Subset}
Let $K$ be a knot with Alexander module~$\mathcal{A}_K \cong \Lambda/(t\unaryminus 2) \oplus \Lambda/(t^{-1}\unaryminus 2)$, and let $D$ be a~$\Lambda$-homology ribbon disc with boundary $K$.
\begin{enumerate}
\item  If $\operatorname{Ord}(P_D)\doteq t\unaryminus 2$, then $P_D \subset  \Lambda/(t\unaryminus 2)  \oplus 0.$
\item If $\operatorname{Ord}(P_D)\doteq  t^{-1}\unaryminus 2$, then $P_D \subset 0\oplus \Lambda/(t^{-1}\unaryminus 2).$
\end{enumerate}
\end{lemma}
\begin{proof}
We only prove the first assertion as the argument for the second is analogous.
We first claim that if $\operatorname{Ord}(P_D)\doteq  t\unaryminus 2$,  then $t\unaryminus 2$ annihilates $P_D$.
Lemma~\ref{lem:HomologyRibbon} implies that~$P_D$ is torsion and has projective dimension $1$,  and therefore admits a square presentation matrix.
It is known that if a torsion~$\Lambda$-module~$H$ admits a square presentation matrix,  then~$\operatorname{Ord}(H)$ annihilates~$H$~\cite[Remark 2 on page 31]{HillmanAlexanderIdeals}.
Thus $\operatorname{Ord}(P_D)\doteq  t\unaryminus 2$ annihilates $P_D$,  as claimed.

The lemma now follows from the same argument as in~\cite[proof of Lemma 4.3 (2)]{ConwayPowellDiscs}.
We repeat the details for the reader's convenience.
Given $([p_1],[p_2]) \in P_D \leq \mathcal{A}_K$, our aim is to prove that~$[p_2]=0 \in \Lambda/(t^{-1}\unaryminus 2)$.
Using the claim, we deduce that~$(t\unaryminus 2)([p_1],[p_2])=0$ and in particular~$[(t\unaryminus 2)p_2]=0$ in~$\Lambda/(t^{-1}\unaryminus 2)$.
This implies that~$(t\unaryminus 2)p_2=(t^{-1}\unaryminus 2)x$ for some~$x \in \Lambda$.
Since~$\Lambda$ is an UFD and since~$(t\unaryminus 2)$ and~$(t^{-1}\unaryminus 2)$ are coprime, we deduce that~$p_2=(t^{-1}\unaryminus 2)z$ for some~$z \in \Lambda$.
It follows that~$[p_2]=0 \in \Lambda/(t^{-1}\unaryminus 2)$ as required. 
\end{proof}

What follows is the first main technical proposition of this section.

\begin{proposition}
\label{prop:GoalTechnical}
If a knot~$K$ is the equatorial cross-section of a sphere $S=D_1 \cup_K D_2 \subset S^4$ with~$\pi_1(X_S) \cong \Z$ and~$\mathcal{A}_K \cong \Lambda/(t\unaryminus 2) \oplus \Lambda/(t^{-1}\unaryminus 2)$, then 
$$\lbrace P_{D_1},P_{D_2} \rbrace=\lbrace  \Lambda/(t\unaryminus 2)  \oplus 0, 0 \oplus  \Lambda/(t^{-1}\unaryminus 2) \rbrace$$
and, for $i=1,2$,  there is a group isomorphism
$$\pi_{D_i}/\pi_{D_i}^{(2)} \cong BS(1,2).$$
\end{proposition}
\begin{proof}
Lemma~\ref{lem:HomologyRibbon} implies that the inclusions induce an isomorphism~$P_{D_1} \oplus P_{D_2} \cong \mathcal{A}_K$.
It follows that $\lbrace \operatorname{Ord}(P_{D_1}),\operatorname{Ord}(P_{D_2})\rbrace \doteq  \lbrace t\unaryminus 2,t^{-1}\unaryminus 2\rbrace.$
We can assume without loss of generality that~$\operatorname{Ord}(P_{D_1})\doteq  t\unaryminus 2$ and  $\operatorname{Ord}(P_{D_2})\doteq  t^{-1}\unaryminus 2$.
We prove that $P_{D_1}=\Lambda/(t\unaryminus 2) \oplus 0$; the proof for the second summand is entirely similar.

The argument is a slight modification of~\cite[proof of Lemma 4.3 (2)]{ConwayPowellDiscs}.
Set $P_1:=\Lambda/(t\unaryminus 2) \oplus 0$ and $P_2:=0 \oplus \Lambda/(t^{-1}\unaryminus 2)$.
Lemma~\ref{lem:Subset} implies that $P_{D_1} \subset P_1$.
Using the short exact sequence from Lemma~\ref{lem:HomologyRibbon} we have
$$\mathcal{A}_{D_1} \cong \mathcal{A}_K/P_{D_1}=P_1/P_{D_1} \oplus P_2.$$
Note that~$\operatorname{Ord}(P_{D_1})\doteq  t\unaryminus 2\doteq  \operatorname{Ord}(P_1)$.
As $\Lambda_\Q:=\Lambda \otimes_\Z \Q$ is a PID, we deduce that~$P_1/P_{D_1} \otimes_{\Z} \Q=0$.
It follows that $P_1/P_{D_1}$ is $\Z$-torsion, so either $\mathcal{A}_{D_1}$ is $\Z$-torsion or $P_1/P_{D_1}=0.$

Using the short exact sequence from Lemma~\ref{lem:HomologyRibbon} and the fact that $\mathcal{A}_K$ is $\Z$-torsion free~\cite[Proposition 3.5]{LevineKnotModules},  we deduce that~$\mathcal{A}_{D_1}$ is $\Z$-torsion free.
It follows that~$P_1/P_{D_1}=0$ and therefore~$P_{D_1}=P_1$.
The same argument shows that  $P_{D_2}=P_2$.

Finally we show that $\pi_{D_i}/\pi_{D_i}^{(2)} \cong BS(1,2).$
We already saw in (the proof of) Lemma~\ref{lem:DoublySlice} that the inclusions induce isomorphisms~$i_{D_2} \colon P_{D_1} \xrightarrow{\cong} \mathcal{A}_{D_2}$ and~$i_{D_1} \colon P_{D_2} \xrightarrow{\cong} \mathcal{A}_{D_1}$.
It follows that~$\mathcal{A}_{D_i} \cong \Z[\tmfrac{1}{2}]$ and so~$\pi_{D_i}/\pi_{D_i}^{(2)} \cong \Z \ltimes \mathcal{A}_{D_i} \cong BS(1,2)$ as required.
This concludes the proof of the proposition.
\end{proof}

Finally, we derive our second technical proposition.

\begin{proposition}
\label{prop:BS12}
Let $K$ be a knot with Alexander module~$\mathcal{A}_K \cong \Lambda/(t\unaryminus 2) \oplus \Lambda/(t^{-1}\unaryminus 2)$.
If $K$ is the equatorial cross-section of a sphere $S=D_1 \cup_K D_2 \subset S^4$ with $\pi_1(N_{D_i})$ metabelian for $i=1,2$ and~$\pi_1(X_S) \cong \Z$, then the discs~$D_1,D_2$ are $G$-homotopy ribbon.
\end{proposition}
\begin{proof}
Proposition~\ref{prop:GoalTechnical} implies that~$\pi_1(N_{D_i})/\pi_1(N_{D_i})^{(2)} \cong G$.
Since we assumed that the~$\pi_1(N_{D_i})$ are metabelian, it follows that $\pi_1(N_{D_i}) \cong  G$.
Again using that the~$\pi_1(N_{D_i})$ are metabelian, the combination of Lemma~\ref{lem:DoublySlice} and Proposition~\ref{prop:HomotopyRibbon} implies that the $D_i$ are homotopy ribbon.
\end{proof}

\subsection{Proof of the main theorem}
\label{sub:ProofMain}

We prove Theorem~\ref{thm:Converse} from the introduction.

\begin{theorem}
\label{thm:VanKampen}
Let~$K$ be a knot with Alexander module~$\mathcal{A}_K \cong \Lambda/(t\unaryminus 2)  \oplus \Lambda/(t^{-1}\unaryminus 2)$.
The following statements are equivalent:
\begin{enumerate}
\item The knot~$K \subset S^3$ arises as the equatorial cross-section of a~$2$-sphere~$S \subset S^4=D^4_1 \cup_{S^3} D^4_2$ with $\pi_1(S^4 \setminus S) \cong \Z$,  and the group of the disc~$D_i=D^4_i \cap S$ is metabelian for $i=1,2.$
\item The summands~$P_1=\Lambda/(t\unaryminus 2) \oplus 0$ and~$P_2=0 \oplus \Lambda/(t^{-1}\unaryminus 2)$ of~$\mathcal{A}_K$ both satisfy the Ext condition.
\end{enumerate}
\end{theorem}
\begin{proof}
We prove that~$(2) \Rightarrow (1)$.
Assume that the summands $P_1$ and $P_2$ both satisfy the Ext condition.
Set $\pi_K:=\pi_1(M_K)$ and fix a splitting $\theta \colon \Z \to \pi_K/\pi_K^{(2)}$ of the short exact sequence in~\eqref{eq:SESIntro}.
As explained in Section~\ref{sub:Ext}, this splitting leads to an isomorphism $\vartheta_K \colon \pi_K/\pi_K^{(2)} \to \Z \ltimes \mathcal{A}_K$ and we can then consider the epimorphisms
\begin{align*}
&\phi_{P_1} \colon \pi_K \stackrel{p}{\twoheadrightarrow} \pi_K/\pi_K^{(2)} \xrightarrow{\vartheta_K,\cong} \Z \ltimes \mathcal{A}_K \xrightarrow{\operatorname{proj}_1} \Z \ltimes \mathcal{A}_K/P_1= \Z \ltimes P_2, \\
&\phi_{P_2} \colon \pi_K \stackrel{p}{\twoheadrightarrow}  \pi_K/\pi_K^{(2)} \xrightarrow{\vartheta_K,\cong} \Z \ltimes \mathcal{A}_K 
\xrightarrow{\operatorname{proj}_2}
 \Z \ltimes \mathcal{A}_K/P_2= \Z \ltimes P_1.
\end{align*} 
Technically speaking, $\phi_{P_i}$ depends on the choice of the splitting $\theta \colon \Z \to \pi_K/\pi_K^{(2)}$ 
but the Ext condition only depends on $P_i$ and $K$; recall Lemma~\ref{lem:ExtWellDef}.
As we recalled in Section~\ref{sub:Classification},  the Ext condition ensures the existence of~$G$-homotopy ribbon discs $D_1$ and $D_2$ with boundary $K$ and such that the inclusion induced map~$\pi_1(M_K) \to \pi_1(N_{D_i})$ agrees with~$\phi_{P_i}$ for $i=1,2$ up to an automorphism of $G=BS(1,2)$.
In order to cut down on notation we will omit these automorphisms from the notation and allow ourselves to identify $\pi_1(N_{D_1})$ with $\Z \ltimes P_2$ (resp.~$\pi_1(N_{D_2})$ with~$\Z \ltimes P_1$) and $\phi_{P_i}$ with~$\pi_1(M_K) \to \pi_1(N_{D_i})$ for $i=1,2$.
We also set $E_K:=S^3 \setminus \nu(K)$.

Since $G=BS(1,2)$ is metabelian,  $D_1$ and $D_2$ have metabelian disc groups and the proof reduces to showing that the $2$-sphere $S=D_1 \cup_K D_2 \subset S^4$ has infinite cyclic knot group.

With the notation introduced up to this point, we obtain the following commutative diagram in which all unlabelled maps are inclusion induced, except for the bottom maps that are defined by combining $\id_{\Z}$ with the apppropriate canonical projection:
$$
\xymatrix{
\pi_1(N_{D_1})\ar[d]^=&\ar[l]\pi_1(E_K)\ar[r]\ar@{->>}[d]&\pi_1(N_{D_2})\ar[d]_= \\
\pi_1(N_{D_1}\ar[d]^=)&\ar[l]\pi_K\ar[r]\ar@{->>}[d]^p&\pi_1(N_{D_2})\ar[d]_= \\
\pi_1(N_{D_1})/\pi_1(N_{D_1})^{(2)}\ar[d]^\cong_{\vartheta_{D_1}}&\ar[l]\pi_K/\pi_K^{(2)}\ar[d]_\cong^{\vartheta_{K}}\ar[r]&\pi_1(N_{D_2})/\pi_1(N_{D_2})^{(2)} \ar[d]_\cong^{\vartheta_{D_2}} \\
\Z \ltimes \mathcal{A}_{D_1} \ar[d]^\cong&\ar[l]_-{\proj_1} \Z \ltimes \mathcal{A}_K \ar[d]_\cong\ar[r]^-{\proj_2}&\Z \ltimes \mathcal{A}_{D_2} \ar[d]_\cong \\
\Z \ltimes P_2&\ar[l] \Z \ltimes (P_1 \oplus P_2)\ar[r]&\Z \ltimes P_1.
}
$$
Van Kampen's theorem implies that $\pi_1(X_S)$ is isomorphic to the pushout of the top row.
Since each of the outer vertical maps is an isomorphism and each of the central vertical maps is either a  surjection or an isomorphism,  this pushout is isomorphic to the push out of the bottom row.\footnote{
Given pushout squares
$$
\xymatrix@R0.5cm@C0.5cm{
A_1 \ar[r]^{f_2}\ar[d]_{f_3}&A_2\ar[d]^{g_2} \\ 
A_3 \ar[r]^{g_3}&A_4} 
\quad  \quad  \quad 
\xymatrix@R0.5cm@C0.5cm{
B_1 \ar[r]^{f_2'}\ar[d]_{f_3'} &B_2\ar[d]^{g_2'}\\ 
B_3 \ar[r]^{g_3'}&B_4,} 
$$
and a commutative diagram
$$
\xymatrix@R0.5cm@C0.5cm{
A_2 \ar[d]^\cong_{\varphi_2}&A_1\ar@{->>}[d] \ar[l]_{f_2}\ar[r]^{f_3} &A_3\ar[d]_\cong^{\varphi_3}\\ 
B_2 &B_1\ar[l]_{f_2'}\ar[r]^{f_3'} &B_3,}
$$
a verification involving the universal property of pushouts ensures that~$ A_2 \xrightarrow{g_2' \circ \varphi_2} B_4 \xleftarrow{g_3' \circ \varphi_3} A_3$
is a pushout of~$A_2  \xleftarrow{f_2}  A_1 \xrightarrow{f_3} A_3.$
 }
Since the push out of the bottom row is isomorphic to $\Z$, we conclude that $\pi_1(X_S) \cong \Z$, as required.
This concludes the proof that~$(2) \Rightarrow (1)$.

\medbreak

We now prove the converse.
We assume that $K$ arises as the equatorial cross-section of a $2$-sphere~$S=D_1 \cup_K D_2$ with~$D_1,D_2 \subset D^4$ two discs with metabelian disc groups and $\pi_1(X_S)\cong \Z$.
Our goal is to prove that the summands of~$\mathcal{A}_K \cong \Lambda/(t\unaryminus 2) \oplus \Lambda/(t^{-1}\unaryminus 2)$ satisfy the Ext condition.
Using the assumption on $\mathcal{A}_K$ and the fact that $\pi_1(X_S)\cong \Z$, Proposition~\ref{prop:GoalTechnical} implies that 
$$\lbrace P_{D_1},P_{D_2} \rbrace=\lbrace  \Lambda/(t\unaryminus 2)  \oplus 0, 0 \oplus  \Lambda/(t^{-1}\unaryminus 2) \rbrace=:\lbrace P_1,P_2 \rbrace.$$
Since we assumed that the disc groups are metabelian, Proposition~\ref{prop:BS12} implies that~$D_1$ and~$D_2$ are~$G$-homotopy ribbon.
Since~$D_1$ and~$D_2$ are~$G$-homotopy ribbon, we know from~\cite[Theorem~1.6]{ConwayPowellDiscs} that the submodules~$P_{D_1}$ and $P_{D_2}$ satisfy the Ext condition.
It follows that~$P_1$ and~$P_2$ also satisfy the Ext condition.
This concludes the proof of the theorem.
\end{proof}


\section{Satellite knots}
\label{sec:Application}

Continuing with the  notation from Section~\ref{sub:ApplicationIntro},  we write $R_\eta(K)$ for the satellite knot with pattern~$R \subset S^3$,  infection curve~$\eta \subset S^3 \setminus \nu R=:E_R$ and companion $K$.
In Section~\ref{sub:ApplyMain} we prove Theorem~\ref{thm:NewDoublySliceIntro} whereas in Section~\ref{sub:NonNecessity}, we prove that the double Ext condition from Theorem~\ref{thm:MainIntro} is not necessary for a knot $J$ with $\mathcal{A}_J \cong \Lambda/(t\unaryminus 2) \oplus \Lambda/(t^{-1}\unaryminus 2)$ to be doubly slice.

\subsection{An application of Theorem~\ref{thm:MainIntro}}
\label{sub:ApplyMain}

We prove Theorem~\ref{thm:NewDoublySliceIntro} from the introduction which gives conditions on a pattern $R$ with $\mathcal{A}_R \cong \Lambda/(t\unaryminus 2) \oplus  \Lambda/(t^{-1}\unaryminus 2)$ so that $R_\eta(K)$ is doubly slice for every companion~$K$.

\begin{construction}
\label{cons:MapFromSatellite}
We construct a degree one map~$f \colon M_{R_\eta(K)} \to M_R$.
The $0$-framed surgery of $R_\eta(K)$ decomposes as~$M_{R_\eta(K)}=M_R \setminus \nu (\eta) \cup E_K$ where the gluing identifies the~$0$-framed longitude of~$K$ with a meridian of~$\eta$ and a meridian of~$K$ with (a $0$-framed longitude of)~$\eta$.
These identifications determine a homeomorphism~$\partial E_K \to \partial \overline{\nu}(\eta)$.
Thinking of~$ \overline{\nu}(\eta)$ as an unknot exterior, this homeomorphism extends to a degree one map~$f_0 \colon E_K \to \overline{\nu}(\eta)$; see e.g.~\cite[Construction~7.1]{MillerPowell} for details.
This map~$f_0$ can be combined with the identity map on~$M_R \setminus \nu(\eta)$  to obtain the required degree one map
$$f=\id \cup f_0 \colon M_{R_\eta(K)} \to M_R.$$
Technically speaking,~\cite[Construction 7.1]{MillerPowell} shows that~$f_0$ combines with the identity map on~$E_R \setminus \nu(\eta)$ to yield a degree one map $E_{R_\eta(K)} \to E_R$.
This is equivalent to noting that~$f$ has degree one because,  for any knot $J$,  excision and the fact that $E_J$ is a homology circle imply that~$H_3(E_J,\partial E_J) \cong  H_3(M_J,E_J) \cong  H_3(M_J).$
\end{construction}

\begin{remark}
\label{rem:DegreeOne}
Since $f$ has degree one,  it induces a surjection on fundamental groups.
In particular if $\psi \colon \pi_1(M_R) \twoheadrightarrow \Gamma$ is an epimorphism, then so is $\psi \circ f_*$.
\end{remark}

The following result is essentially contained in~\cite[Lemma 6.2]{FriedlTeichner}.
Under the appropriate assumptions, the statement of~\cite[Lemma 6.2]{FriedlTeichner} does not describe an explicit isomorphism~$H_1(M_{R_\eta(K)};\Z[\Gamma]_{\psi \circ f_*}) \xrightarrow{\cong} H_1(M_R;\Z[\Gamma]_\psi)$ but a close look at the proof shows that it is induced by~$f$; we recall the relevant details.

\begin{lemma}
\label{lem:IsoHomology}
The map~$f \colon M_{R_\eta(K)} \to M_R$ induces an isomorphism~$H_1(M_{R_\eta(K)}) \to H_1(M_R)$ on homology and,  given an epimorphism~$\psi \colon \pi_1(M_R) \twoheadrightarrow \Gamma$, if either~$\psi(\eta)=1$ or $\Delta_K \doteq 1$, then~$f$ induces a~$\Z[\Gamma]$-isomorphism
$$  
f_* \colon H_1(M_{R_\eta(K)};\Z[\Gamma]_{\psi \circ f_*}) \xrightarrow{\cong} H_1(M_R;\Z[\Gamma]_\psi).
$$
\end{lemma}
\begin{proof}
By construction, the map~$f$ takes a meridian of~$R_\eta(K)$ to a meridian of~$R$  and therefore induces an isomorphism on first homology. 

The second assertion will be proved by comparing the Mayer-Vietoris sequences for the decompositions~$M_{R_\eta(K)}=M_R \setminus \nu(\eta) \cup E_K$ and~$M_{R}=M_R \setminus \nu(\eta) \cup \overline{\nu}(\eta)$ and applying the $5$-lemma.

We proceed with the details.
Consider the following commutative diagram in which the rows are exact, with the top (resp. bottom) row having coefficients in $\Z[\Gamma]_{\psi \circ f_*}$  (resp. $\Z[\Gamma]_{\psi}$):
$$
\xymatrix@R0.5cm@C0.3cm{
H_1(\partial E_K) \ar[r]\ar[d]^{f_*,\cong}&
H_1(M_R \setminus \nu(\eta)) \oplus H_1(E_K) \ar[r]\ar[d]^{f_*}&
H_1(M_{R_\eta(K)})\ar[r] \ar[d]^{f_*}& 
H_0(\partial E_K) \ar[r]\ar[d]^{f_*,\cong}&
H_0(M_R \setminus \nu(\eta)) \oplus H_1(E_K)\ar[d]^{f_*}& \\
H_1(\partial \overline{\nu}(\eta))  \ar[r]&
H_1(M_R \setminus \nu(\eta)) \oplus H_1(\overline{\nu}(\eta))  \ar[r]&
H_1(M_{R}) \ar[r]&
H_0(\partial \overline{\nu}(\eta))  \ar[r]&
H_0(M_R \setminus \nu(\eta)) \oplus H_1(\overline{\nu}(\eta)).&
}
$$
Here we used that $f \colon \partial E_K \to \partial \overline{\nu}(\eta)$ is an homeomorphism and that,  by definition of $f=\id \cup f_0$, we know that~$f_* \colon H_*(M_{R_\eta(K)},\Z[\Gamma]_{\psi \circ f_*}) \to H_*(M_R;\Z[\Gamma]_{\psi})$ is an isomorphism.
In order to apply the~$5$-lemma we therefore prove that if either~$\psi(\eta)=1$ or $\Delta_K \doteq 1$, then~$f \colon E_K \to \overline{\nu}(\eta)$ induces an isomorphism on $H_i(-;\Z[\Gamma])$ for $i=0,1$.

Assume that~$\psi(\eta)=1$. Since (a $0$-framed longitude of)~$\eta$ is identified with a meridian~$\mu_K$ of~$K$ in~$M_{R_\eta(K)}$,  we have~$\psi \circ f_*(\mu_K)=1$.
It follows that $X \in \lbrace \overline{\nu}(\eta), E_K \rbrace$ is endowed with the trivial coefficient system and therefore $H_*(X;\Z[\Gamma]) \cong H_*(X) \otimes_\Z \Z[\Gamma]).$
Since $f \colon E_K \to \overline{\nu}(\eta)$ induces an isomorphism on $\Z$-homology, it then also induces an isomorphism on $\Z[\Gamma]$-homology.

We now drop the assumption on~$\psi(\eta)$.
The composition $\pi_1(E_K) \to \pi_1(M_{R_\eta(K)}) \xrightarrow{f} \pi_1(M_R)$ factors through $\pi_1(\overline{\nu}(\eta)) \cong \Z$.
It follows that for $X \in \lbrace \overline{\nu}(\eta), E_K \rbrace$, the coefficient system factors through abelianisation and therefore $H_*(X;\Z[\Gamma]) \cong H_*(X;\Lambda \otimes_\Lambda \Z[\Gamma]).$
Since $\Delta_K \doteq 1$ we have~$H_1(E_K;\Lambda)=0$.
It follows that for $X \in \lbrace \overline{\nu}(\eta), E_K \rbrace$ we have $H_i(X;\Lambda)=0$ for $i=1,2$ and the universal coefficient spectral sequence shows that $H_0(X;\Z[\Gamma]) \cong H_0(X;\Lambda) \otimes_\Lambda \Z[\Gamma]$ and~$H_1(X;\Z[\Gamma]) \cong \operatorname{Tor}_1^\Lambda(H_0(X;\Lambda),\Z[\Gamma])$.
As $f \colon E_K \to \overline{\nu}(\eta)$ induces an isomorphism on~$H_0(-;\Lambda)$, it therefore also induces an isomorphism on $H_i(-;\Z[\Gamma])$ for $i=0,1$.

Thus if either~$\psi(\eta)=1$ or $\Delta_K \doteq 1$, then we can apply the $5$-lemma to the aforementioned commutative diagram of Mayer-Vietoris exact sequences and deduce that $f \colon M_{R_\eta(K)} \to M_R$ induces an isomorphism on $H_1(-;\Z[\Gamma])$.
This concludes the proof of the lemma.
\end{proof}

Continuing with the notation introduced above,  but additionally setting $\pi_R:=\pi_1(M_R)$ and $\pi_{R_\eta(K)}:=\pi_1(M_{R_\eta(K)})$,  the next lemma serves to relate the coefficient systems on~$M_R$ and~$M_{R_\eta(K)}$.

\begin{lemma}
\label{lem:GroupAutomorphism}
Assume that~$\mathcal{A}_R\cong\Lambda/(t\unaryminus 2) \oplus \Lambda/(t^{-1}\unaryminus 2)$.
Let~$P \subset \mathcal{A}_R$ be one of the two summands and let~$\eta \subset S^3 \setminus R$ be a winding number zero infection curve.
The map~$f \colon M_{R_\eta(K)} \to M_R$ from Construction~\ref{cons:MapFromSatellite} gives rise to a group automorphism~$f_*^G \colon G \to G$ that makes the following diagram commute:
$$
\xymatrix{
\pi_{R_\eta(K)} \ar@{->>}[r]^-{\phi_{f^{-1}_*(P)}}\ar[d]_{f_*} &G \ar[d]^{f_*^G} \\
\pi_{R} \ar@{->>}[r]^-{\phi_P} &G.
}
$$

\end{lemma}
\begin{proof}
Since~$\eta$ has winding number zero, it belongs to the kernel of the abelianisation map~$\operatorname{ab} \colon \pi_R \twoheadrightarrow \Z$.
Lemma~\ref{lem:IsoHomology} applied with~$\psi=\operatorname{ab}$ implies that~$f$ induces an isomorphism on first homology and on the Alexander modules.
It follows that~$f$ also induces an isomorphism on the second derived quotients:
\begin{equation}
\label{eq:SecondDerivedHomomf}
\xymatrix@R0.5cm@C0.5cm{
1\ar[r]& 
\mathcal{A}_{R_\eta(K)} \ar[r]\ar[d]^{f_*}_\cong&
 \pi_R/\pi_{R_\eta(K)}^{(2)} \ar[r] \ar[d]^{f_*}_\cong
 & H_1(M_{R_\eta(K)}) \ar[d]^{f_*}_\cong \ar[r]
 &1 \\
1\ar[r]& \mathcal{A}_R \ar[r]& \pi_R/\pi_R^{(2)} \ar[r]
&  H_1(M_R)  \ar[r]&1.
}
\end{equation}
Choose meridians $\mu_R \in \pi_R$ and $\mu_{R_\eta(K)}\in \pi_{R_\eta(K)}$ such that $f_*(\mu_R)=\mu_{R_\eta(K)}$.
These choices lead to splittings  of the rows in this diagram and to isomorphisms $\vartheta_R \colon \pi_R/\pi_R^{(2)} \to \Z \ltimes \mathcal{A}_R$ and $\vartheta_{R_\eta(K)} \colon \pi_{R_\eta(K)}/\pi_{R_\eta(K)}^{(2)} \to \Z \ltimes \mathcal{A}_{R_\eta(K)}$ that fit into the following commutative diagram:
$$
\xymatrix@R0.5cm{
\pi_{R_\eta(K)} \ar[r]\ar[d]^{f_*}&
\pi_{R_\eta(K)}/\pi_{R_\eta(K)}^{(2)}\ar[r]^{\vartheta_{R_\eta(K)},\cong} \ar[d]^{f_*}_\cong&
\Z \ltimes \mathcal{A}_{R_\eta(K)}\ar[r]\ar[d]^{f_*}_\cong& 
 \Z \ltimes \mathcal{A}_{R_\eta(P)}/f^{-1}_*(P)\ar[d]^{f_*}_\cong \\
\pi_R\ar[r]&
\pi_R/\pi_R^{(2)}\ar[r]^{\vartheta_R,\cong}&
\Z \ltimes \mathcal{A}_R\ar[r]&
 \Z \ltimes \mathcal{A}_R/P.
}
$$
Since we assumed that~$\mathcal{A}_R\cong\Lambda/(t\unaryminus 2) \oplus \Lambda/(t^{-1}\unaryminus 2)$,  the two rightmost terms are isomorphic to~$G=BS(1,2)$ (as abelian groups) and this is how we define the automorphism~$f_*^G$.
This concludes the proof of the lemma.
\end{proof}

Note that~$f_*^G$ depends on the choice of splittings of the rows in~\eqref{eq:SecondDerivedHomomf} and on the choice of the identification of the righmost terms in the previous diagram with $G$.
These choices do not affect the remainder of the argument.

The following proposition is similar to~\cite[Proposition 7.4]{FriedlTeichner}.

\begin{proposition}
\label{prop:ExtConditionImage}
Let~$R$ be a pattern with~$\mathcal{A}_R\cong\Lambda/(t\unaryminus 2) \oplus \Lambda/(t^{-1}\unaryminus 2)$, let~$P \subset \mathcal{A}_R$ be one of the two summands, and let~$\eta \subset S^3 \setminus R$ be a winding number zero infection curve.
Assume that either~$\phi_P(\eta)=1$ or~$\Delta_K \doteq 1$.
If~$P \leq \mathcal{A}_R$ satisfies the Ext condition,  then~$f_*^{-1}(P) \leq \mathcal{A}_{R_\eta(K)}$ satisfies the Ext condition.
\end{proposition}
\begin{proof}
Lemma~\ref{lem:TwistedHomology} implies that the automorphism~$f^G_* \colon G \to G$ from Lemma~\ref{lem:GroupAutomorphism} induces an isomorphism $H_1(M_{R_\eta(K)};\Z[G]_{ \phi_{f_*^{-1}(P)}}) \cong H_1(M_{R_\eta(K)};\Z[G]_{f_*^G \circ \phi_{f_*^{-1}(P)}})$.
Here note that Lemma~\ref{lem:GroupAutomorphism} applies because we assumed that $\eta$ has winding number zero.
Since~$\phi_P \circ f_*= f_*^G \circ \phi_{f_*^{-1}(P)}$, we obtain the following sequence of isomorphisms:
\begin{align*}
H_1(M_{R_\eta(K)};\Z[G]_{ \phi_{f_*^{-1}(P)}})
& \xrightarrow{\cong } 
H_1(M_{R_\eta(K)};\Z[G]_{f_*^G \circ \phi_{f_*^{-1}(P)}}) \\
& \xrightarrow{=} 
 H_1(M_{R_\eta(K)};\Z[G]_{\phi_P \circ f_*}) \\
&\xrightarrow{\cong,f_*}
 H_1(M_R;\Z[G]_{\phi_P}).
\end{align*}
The last isomorphism comes from Lemma~\ref{lem:IsoHomology}; this lemma applies because we assumed that either~$\phi_P(\eta)=1$ or $\Delta_K \doteq 1$.

The Ext functor is contravariant in the first variable,  so we obtain an isomorphism
$$\operatorname{Ext}^1_{\Z[G]}(H_1(M_R;\Z[G]_{\phi_P}),\Z[G])
\xrightarrow{\cong} 
\operatorname{Ext}^1_{\Z[G]}(H_1(M_{R_\eta(K)};\Z[G]_{\phi_{f_*^{-1}(P)}}),\Z[G]).
$$
Since~$P$ satisfies the Ext condition, the left hand side is trivial and therefore so is the right hand side.
This shows that~$f_*^{-1}(P)$ satisfies the Ext condition and thus concludes the proof of the proposition.
\end{proof}


We are now ready to prove Theorem~\ref{thm:NewDoublySliceIntro} from the introduction.
We recall the statement for the reader's convenience.


\begin{theorem}
\label{thm:ApplicationMain}
Let~$R$ be a pattern and~$\eta \subset S^3 \setminus R$ be an infection curve that lies in~$\pi_1(S^3 \setminus R)^{(2)}$.
If both summands of~$\mathcal{A}_{R}=\Lambda/(t\unaryminus 2) \oplus \Lambda/(t^{-1}\unaryminus 2)$ satisfy the Ext condition,  then~$R_\eta(K)$ is doubly slice for any~$K$.
\end{theorem}
\begin{proof}
We verify that~$R_\eta(K)$ has Alexander module~$\Lambda/(t\unaryminus 2) \oplus \Lambda/(t^{-1}\unaryminus 2)$ and that its summands satisfy the Ext condition.
Since~$\eta \in \pi_1(E_R)^{(2)}$,  it lies in~$\pi_1(E_R)^{(1)}$ i.e.  $\eta$ belongs to the kernel of the abelianisation map $\operatorname{ab} \colon \pi_1(E_R) \twoheadrightarrow \Z$.
It follows that $\eta \in \ker(\operatorname{ab} \colon \pi_1(M_R) \twoheadrightarrow \Z)$ and we can therefore apply Lemma~\ref{lem:IsoHomology} with $\psi=\operatorname{ab}$  to deduce that~$f \colon M_{R_\eta(K)} \to M_R$ induces a~$\Lambda$-isomorphism on the Alexander modules:
$$f_* \colon  \mathcal{A}_{R_\eta(K)} \xrightarrow{\cong} \mathcal{A}_R \cong \Lambda/(t\unaryminus 2) \oplus \Lambda/(t^{-1}\unaryminus 2).$$
Use $P_1$ and $P_2 $ to denote the summands of $\mathcal{A}_R$ so that~$f_*^{-1}(P_1)$ and~$f_*^{-1}(P_2)$ are the summands of~$\mathcal{A}_{R_\eta(K)}.$
By definition,  $\phi_{P_i}$ factors through $\pi_1(M_R)/\pi_1(M_R)^{(2)}$ and thus vanishes on the second derived subgroup of $\pi_1(M_R)$. 
In particular we have $\phi_{P_i}(\eta)=~1$.
Since both summands~$P_1,P_2$ of~$\mathcal{A}_R$ satisfy the Ext condition and $\phi_{P_1}(\eta)=1=\phi_{P_2}(\eta)$, Proposition~\ref{prop:ExtConditionImage} ensures that~$f_*^{-1}(P_1)$ and~$f_*^{-1}(P_2)$ also satisfy the Ext condition.
Theorem~\ref{thm:MainIntro} implies that~$R_\eta(K)$ is doubly slice.
\end{proof}


We apply this theorem to a concrete example.

\begin{figure}[!htbp]
\centering
\captionsetup[subfigure]{}
\begin{subfigure}[t]{0.3\textwidth}
\includegraphics[scale=0.3]{946Satellite1Hand}
\end{subfigure}
 \hspace{2.5cm}
\begin{subfigure}[t]{.30\textwidth}
\includegraphics[scale=0.3]{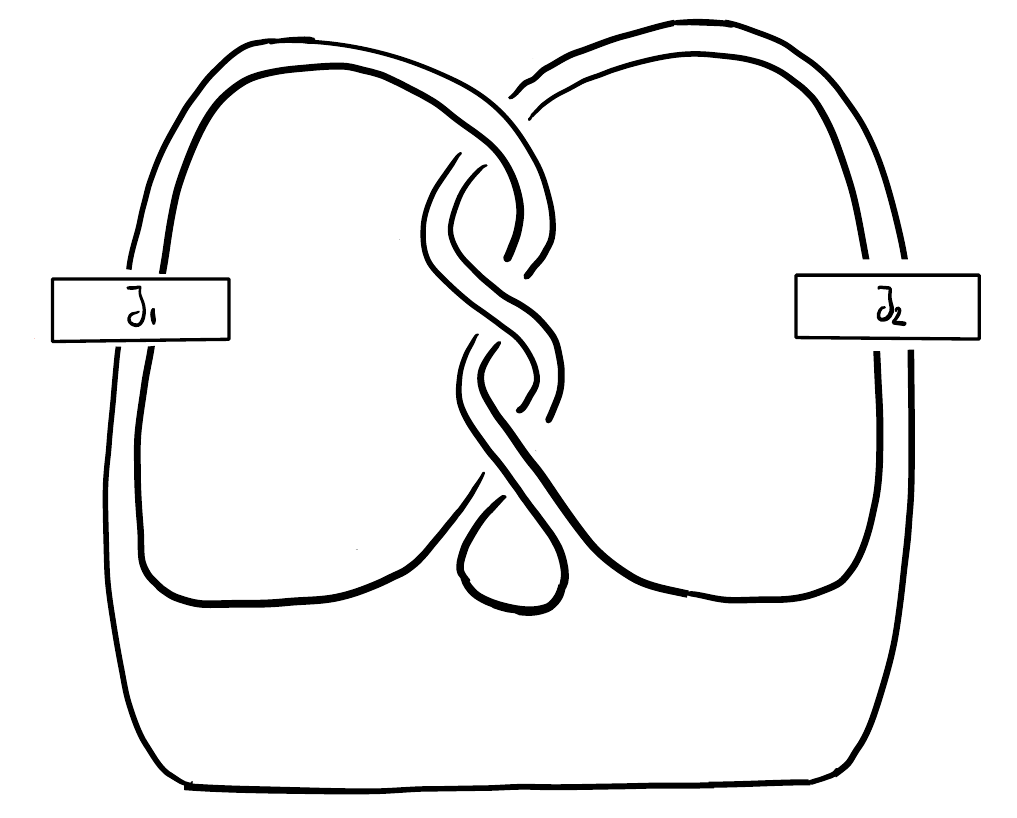}
\end{subfigure}
\caption{On the left: the knot $\mathcal{R}:=9_{46}$ with the infections curves $\gamma_1,\gamma_2$; on the right: the satellite knot $\mathcal{R}(J_1,J_2)$ obtained by infecting $\mathcal{R}$ along the curves~$\gamma_1,\gamma_2$ with respective companions $J_1,J_2$.}
\label{fig:946Satellite}
\end{figure}

\begin{construction}
\label{cons:NewExample}
Consider the knot~$\mathcal{R}=9_{46}$ together with the infection curves~$\gamma_1$ and $\gamma_2$ illustrated on the left of Figure~\ref{fig:946Satellite} and write $\mathcal{R}(J_1,J_2)$ for the satellite knot obtained by infecting~$\mathcal{R}$ along~$\gamma_1,\gamma_2$ with respective companions $J_1$ and $J_2$.
The knot $\mathcal{R}(J_1,J_2)$ is illustrated on the right hand side  of Figure~\ref{fig:946Satellite}.
For $i=1,2$ write~$W_i:=\operatorname{Wh}(T_i)$ for an untwisted Whitehead double of a knot~$T_i$ (with either choice of clasp)
and set~$\mathcal{R}^{T_1,T_2}:=\mathcal{R}(W_1,W_2)$.
Finally consider the curves $\eta_1,\eta_2 \subset S^3 \setminus \mathcal{R}^{T_1,T_2}$ depicted in Figure~\ref{fig:Repeat} and write $ \mathcal{R}^{T_1,T_2}(K_1,K_2)$ for the knot obtained by infecting~$\mathcal{R}^{T_1,T_2}$ along~$\eta_1,\eta_2$ with respective companions $K_1$ and $K_2$.
\end{construction}

\begin{figure}[!htbp]
\centering
\captionsetup[subfigure]{}
\begin{subfigure}[t]{0.3\textwidth}
\includegraphics[scale=0.3]{ExampleDoublySlice}
\end{subfigure}
 \hspace{2.5cm}
\begin{subfigure}[t]{.30\textwidth}
\includegraphics[scale=0.3]{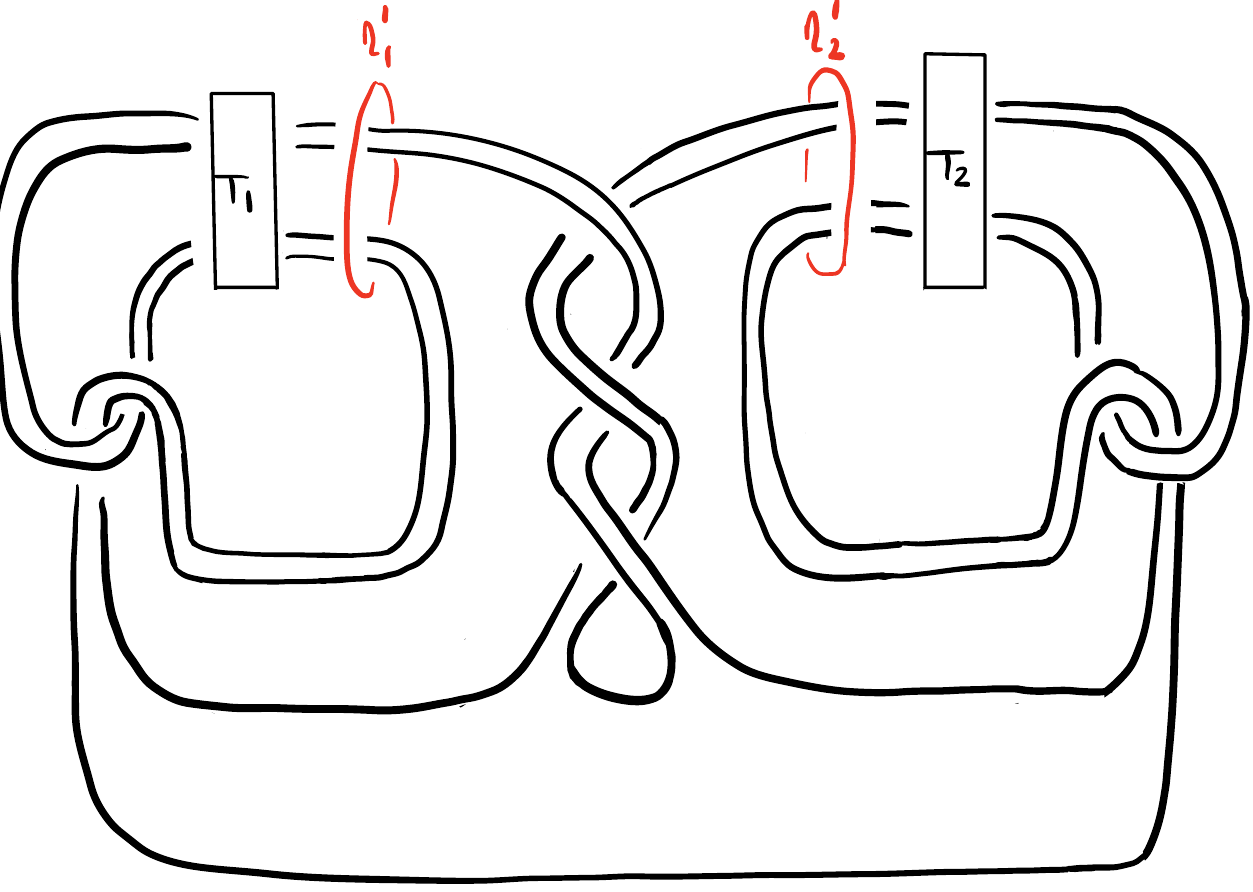}
\end{subfigure}
\caption{On the left: the knot $\mathcal{R}^{T_1,T_2}$ together with the infection curves $\eta_1,\eta_2$; on the right: the same knot but with the infection curves $\eta_1',\eta_2'$.}
\label{fig:Repeat}
\end{figure}


\begin{proposition}
\label{prop:Application}
The knot $\mathcal{R}^{T_1,T_2}(K_1,K_2)$ is doubly slice for any knots $T_1,T_2,K_1,K_2$.
\end{proposition}
\begin{proof}
We abbreviate $\mathcal{R}^{T_1,T_2}$ by $R$.
We claim that the $\eta_i$ lie in the second derived subgroup of~$\pi_1(M_R)$ for $i=1,2.$
It suffices to prove that~$\eta_i \in \pi_1(E_R)^{(2)}$ where~$E_R:=S^3 \setminus \nu(R)$.
Next note that the $\eta_i$ are homotopic in $E_R$ to the curves $\eta_i'$ illustrated on the right of Figure~\ref{fig:Repeat}. 
We will therefore argue that the $\eta_i'$ lie in  $\pi_1(E_R)^{(2)}$. 
Since Whitehead doubling is a winding number zero satellite operator,  the image of any curve $\eta \subset E_{T_i}$ under the inclusion~$E_{T_i} \subset E_{W_i}$ has linking number zero with $W_i$ and thus the image of the map~$\pi_1(E_{T_i}) \to \pi_1(E_{W_i})$ lies in the commutator subgroup of~$\pi_1(E_{W_i})$.
The same argument shows that the image of the inclusion induced map~$\pi_1(E_{W_i}) \to \pi_1(E_R)$
 lies in the commutator subgroup of $\pi_1(E_R)$.
Since $\eta_i'$ belongs to the image of the composition~$\pi_1(E_{T_i}) \to \pi_1(E_{W_i}) \to \pi_1(E_R)$, we deduce that it belongs to~$\pi_1(E_R)^{(2)}$ and thus so does $\eta_i$, 
 proving the claim.

Since $R$ is obtained from $\mathcal{R}=9_{46}$ by winding number zero satellite operations,  Lemma~\ref{lem:IsoHomology} implies that $\mathcal{A}_R \cong \mathcal{A}_{\mathcal{R}} \cong \Lambda/(t\unaryminus 2) \oplus \Lambda/(t^{-1}\unaryminus 2)$.
Since the~$W_i$ have trivial Alexander polynomial,  Lemma~\ref{lem:IsoHomology} implies that $H_1(M_R;\Z[G]) \cong H_1(M_{\mathcal{R}};\Z[G])$.
Since both summands of $\mathcal{A}_{\mathcal{R}}$ satisfy the Ext condition and the $W_i$ have Alexander polynomial one,  Proposition~\ref{prop:ExtConditionImage} shows that both summands of $\mathcal{A}_R$ satisfy the Ext condition.

Since both summands of $\mathcal{A}_R$ satisfy the Ext condition and $\eta_1 \in \pi_1(M_R)^{(2)}$, (the proof of) Theorem~\ref{thm:ApplicationMain} shows that both summands of~$\mathcal{A}_{R_{\eta_1}(K_1)} \cong \Lambda/(t \unaryminus 2) \oplus \Lambda/(t^{-1} \unaryminus 2)$ satisfy the Ext condition.
One can then verify that~$\eta_2 \in \pi_1(M_R)^{(2)}$ remains in the second derived subgroup of~$\pi_1(M_{R_{\eta_1}(K_1)})$: for example because~$R_{\eta_1}(K_1)$ can be thought of as satellite knot with companion~$W_2$,
allowing one to repeat the argument given in the first paragraph of this proof.
Finally since~$\mathcal{A}_{R_{\eta_1}(K_1)} \cong \Lambda/(t \unaryminus 2) \oplus \Lambda/(t^{-1} \unaryminus 2)$ and~$\mathcal{R}^{T_1,T_2}(K_1,K_2)$ is a satellite knot with pattern~$R_{\eta_1}(K_1)$ and infection curve $\eta_2 \in \pi_1(M_{R_{\eta_1}(K_1)})^{(2)}$, a second application of Theorem~\ref{thm:ApplicationMain} shows that~$\mathcal{R}^{T_1,T_2}(K_1,K_2)$ is doubly slice for any knot $K_2$.
\end{proof}

The reason for considering the infection curves $\eta_1,\eta_2$ instead of the curves $\eta_1',\eta_2'$ is that if one used the latter,  then the knot $\mathcal{R}^{T_1,T_2}(K_1,K_2)$ would be isotopic to $\mathcal{R}(\operatorname{Wh}(T_1 \# K_1),\operatorname{Wh}(T_2 \# K_2))$ which is already known to be doubly slice since both $\mathcal{R}$ and Whitehead doubles are doubly slice.


\subsection{The double Ext condition is not necessary for double sliceness}
\label{sub:NonNecessity}

We continue with the notation from the previous section: $\mathcal{R}$ denotes $9_{46}$ and $\mathcal{R}(J_1,J_2)$ denotes the knot illustrated on the right of Figure~\ref{fig:946Satellite}.

\begin{proposition}
\label{prop:NotNecessary}
The knot $K:=\mathcal{R}(\mathcal{R},\mathcal{R})$ has Alexander module $\mathcal{A}_K \cong \Lambda/(t\unaryminus 2)  \oplus \Lambda/(t^{-1}\unaryminus 2)$ and is doubly slice, but its summands do not satisfy the Ext condition.
\end{proposition}
\begin{proof}
We recalled that a satellite knot is doubly slice if both the pattern and the companion are doubly slice; see e.g.~\cite[Proposition 3.4]{Meier} for a proof.
Since $\mathcal{R}=9_{46}$ is doubly slice,  we deduce that the knot $K=\mathcal{R}(\mathcal{R},\mathcal{R})$ is doubly slice.
Since $K$ is obtained by a winding number zero satellite operation on a knot with Alexander module~$\Lambda/(t\unaryminus 2)  \oplus \Lambda/(t^{-1}\unaryminus 2)$, we deduce from Lemma~\ref{lem:IsoHomology} that its Alexander module is also~$\Lambda/(t\unaryminus 2)  \oplus \Lambda/(t^{-1}\unaryminus 2).$

If either of the summands of $\mathcal{A}_K$ satisfied the Ext condition, then by~\cite{ConwayPowellDiscs,FriedlTeichner} $K$ would be~$G$-homotopy ribbon.
This would contradict~\cite[Proposition 7.7]{FriedlTeichner} according to which $\mathcal{R}(J_1,J_2)$ is not~$G$-homotopy ribbon if $\Delta_{J_1} \neq 1$ and $\Delta_{J_2} \neq 1$.
Note that~\cite[Proposition 7.7]{FriedlTeichner} refers to $\mathcal{R}$ as being $6_1$ but as explained in the erratum~\cite{FriedlTeichnerErratum}, it is actually $9_{46}.$
\end{proof}

\def\MR#1{}
\bibliography{BiblioDoubly}
\end{document}